\documentclass[12pt,oneside]{amsart}

\allowdisplaybreaks

\theoremstyle{plain}
\newtheorem{theorem}{Theorem}[section]
\newtheorem{lemma}[theorem]{Lemma}

\newtheorem{conjecture}[theorem]{Conjecture}
\newtheorem{corollary}[theorem]{Corollary}

\theoremstyle{definition}
\newtheorem*{definition}{Definition}

\theoremstyle{remark}

\newcommand{\Z}{\mathbb{Z}}

\newcommand{\N}{\mathbb{N}}
\DeclareMathOperator{\re}{Re}

\begin{document}
\date{\today}

\title{Matrices related to Dirichlet series}

\author{David~A.~Cardon}

\address{Department of Mathematics, Brigham Young University, Provo, UT 84602, USA}

\email{cardon@math.byu.edu}

\maketitle

\begin{abstract}
We attach a certain $n \times n$ matrix $A_n$ to the Dirichlet
series $L(s)=\sum_{k=1}^{\infty}a_k k^{-s}$. We study the
determinant, characteristic polynomial, eigenvalues, and
eigenvectors of these matrices. The determinant of $A_n$ can be
understood as a weighted sum of the first $n$ coefficients of
the Dirichlet series $L(s)^{-1}$. We give an interpretation of
the partial sum of a Dirichlet series as a product of
eigenvalues. In a special case, the determinant of $A_n$ is the
sum of the M\"obius function. We disprove a conjecture of
Barrett and Jarvis regarding the eigenvalues of~$A_n$.
\end{abstract}

\section{Introduction}

To the Dirichlet series
\[
L(s) = \sum_{k=1}^{\infty} \frac{a_k}{k^s},
\]
we attach the $n \times n$ matrix
\[
D_n = \sum_{k=1}^{\infty} a_k E_n(k),
\]
where $E_n(k)$ is the $n \times n$ matrix whose $ij$th entry is
$1$ if $j=ki$ and $0$ otherwise. For example,
\[
D_6 =
\begin{pmatrix}
a_1 & a_2 & a_3 & a_4 & a_5 & a_6 \\
    & a_1 &     & a_2 &     & a_3 \\
    &     & a_1 &     &     & a_2 \\
    &     &     & a_1 &     &     \\
    &     &     &     & a_1 &     \\
    &     &     &     &     & a_1
\end{pmatrix}.
\]
Since
\[
E_n(k_1)E_n(k_2) = E_n(k_1 k_2)
\]
for every $k_1,k_2 \in \N$, formally manipulating linear
combinations of $E_n(k)$ is very similar to formally
manipulating Dirichlet series. However, because $E_n(k)$ is the
zero matrix whenever $k>n$, the sum defining $D_n$ is
guaranteed to converge. Of course, the $n \times n$ matrix
contains less information than the Dirichlet series. Letting
$n$ tend to infinity produces semi-infinite matrices, the
formal manipulation of which is exactly equivalent to formally
manipulating Dirichlet series.

Let $W_n$ be the matrix whose first column is the weight vector
$(0,w_2,w_3,\ldots,w_n)^{T}$ and whose other entries are zeros.
Define the $n \times n$ matrix $A_n$ (and the special cases
$B_n$ and $C_n$) by
\begin{equation} \label{eqn:definitionofABC}
\begin{split}
A_n & = W_n + D_n, \\
B_n & = W_n + D_n \quad \text{when $a_k=1$ for all $k$, }\\
C_n & = W_n + D_n \quad \text{when $a_k=1$ and $w_k=1$ for all $k$.}
\end{split}
\end{equation}
For example, $A_6$, $B_6$, and $C_6$ are the following three
matrices:
\[
\begin{pmatrix}
a_1  & a_2 & a_3 & a_4 & a_5 & a_6 \\
w_2  & a_1 &     & a_2 &     & a_3 \\
w_3  &     & a_1 &     &     & a_2 \\
w_4  &     &     & a_1 &     &     \\
w_5  &     &     &     & a_1 &     \\
w_6  &     &     &     &     & a_1
\end{pmatrix},
\begin{pmatrix}
  1  & 1 & 1 & 1 & 1 & 1 \\
w_2  & 1 &   & 1 &   & 1 \\
w_3  &   & 1 &   &   & 1 \\
w_4  &   &   & 1 &   &   \\
w_5  &   &   &   & 1 &   \\
w_6  &   &   &   &   & 1
\end{pmatrix},
\begin{pmatrix}
1  & 1 & 1 & 1 & 1 & 1 \\
1  & 1 &   & 1 &   & 1 \\
1  &   & 1 &   &   & 1 \\
1  &   &   & 1 &   &   \\
1  &   &   &   & 1 &   \\
1  &   &   &   &   & 1
\end{pmatrix}.
\]
We will always assume that $a_1=1$ since this ensures that the
Dirichlet series $\sum a_k k^{-s}$ has a formal inverse and
since this is true for many Dirichlet series that arise in
number theory. For notational convenience, we set $w_1=1$, and
occasionally we will write $a(i)$ instead of $a_i$. Several
authors have studied the matrices $B_n$ and $C_n$.
In~\cite{Redheffer1977}, it was observed that that
\begin{equation}\label{eqn:Redhefferdeterminant}
\det B_n = \sum_{k=1}^n w_k\mu(k),
\end{equation}
where $\mu$ is the M\"obius $\mu$-function. This is a special
case of the slightly more general fact (see
Theorem~\ref{thm:determinantofAn} below) that
\begin{equation} \label{eqn:determinantofA}
\det A_n = \sum_{k=1}^{n} w_k b_k,
\end{equation}
where the numbers $b_k$ are the coefficients of the formal
series
\[
L(s)^{-1} = \sum_{k=1}^{\infty} \frac{b_k}{k^{s}}.
\]
Thus, $\det A_n$ is a weighted sum of the coefficients of
$L(s)^{-1}$.

To obtain \eqref{eqn:Redhefferdeterminant} from
\eqref{eqn:determinantofA}, choose the Dirichlet series
to be the Riemann zeta function $\zeta(s)=\sum_{k=1}^{\infty}
k^{-s}$ so that $a_k=1$ for all $k$. This corresponds to the
case of the matrix $B_n$. Since
$\zeta(s)^{-1}=\sum_{k=1}^{\infty}
\mu(k)k^{-s}$, where $\mu$ is the M\"obius $\mu$-function, it
follows that $b_k=\mu(k)$. One particularly intriguing choice
for $w_k$ is $w_k=k^{-s}$. Then~\eqref{eqn:determinantofA}
results in the truncated Dirichlet series
\[
\det A_n = \sum_{k=1}^n \frac{b_k}{k^s}.
\]

As the asymptotic growth of sums of the type in
equation~\eqref{eqn:determinantofA} is important to analytic
number theory, representing those sums in terms of determinants
becomes very interesting.

Recall that the Riemann hypothesis is equivalent to the
statement
\[
\sum_{k=1}^n \mu(k) = O(n^{1/2+\epsilon}),
\]
for every positive $\epsilon$. Thus, the Riemann hypothesis is equivalent to
\[
\det C_n =O(n^{1/2+\epsilon}),
\]
for every positive $\epsilon$.

In~\cite{BarretForcadePollington1988}, Barrett, Forcade, and
Pollington expressed the characteristic polynomial of $C_n$ as
\begin{equation}\label{eqn:charpoly}
p_n(x) = (x-1)^{n-r-1} \left((x-1)^{r+1}-\sum_{k=1}^r v(n,k) (x-1)^{r-k}\right),
\end{equation}
where $r=\lfloor \log_2 n \rfloor$ and where the coefficients
$v(n,k)$ were described in terms of directed graphs. We will
refer to the eigenvalue $1$, whose multiplicity is $n-r-1$, as
the \textit{trivial} eigenvalue. The eigenvalues $\lambda
\not=1$ will be called \textit{nontrivial} eigenvalues.
In Theorem~\ref{thm:CharacteristicPolynomial} we extend this
result by determining the characteristic polynomial of the more
general matrix $A_n$. In~\cite{BarretForcadePollington1988}, it
was shown that the spectral radius $\rho(C_n)$ of $C_n$ is
asymptotic to $\sqrt{n}$.

Barrett and Robinson~\cite{RobinsonBarrett1989} determined that
the sizes of the Jordan blocks of $B_n$ corresponding to the trivial
eigenvalue $1$ are
\[
\lfloor\log_2(n/3)\rfloor+1, \lfloor\log_2(n/5)\rfloor+1,\ldots,
\lfloor\log_2(n/\{n\})\rfloor+1,
\]
where $\{n\}$ denotes the greatest odd integer $\leq n$.
Theorem~\ref{thm:eigenvectors} of this paper shows that each
nontrivial eigenvalue of $A_n$ is simple and expresses a basis
for the one-dimensional eigenspace  in terms of a recursion
involving the coefficients of $p_m(x)$ for $m<n$, enhancing our
understanding of the Jordan form of $A_n$.
Theorem~\ref{thm:eigenvectorstranspose} gives a similar result
for the transpose of $A_n$.

The coefficients of the characteristic polynomial of $C_n$ are
related to the Riemann zeta function as follows: If
$(\zeta(s)-1)^k$ is expressed as a Dirichlet series
$\sum_{m=1}^{\infty} \frac{c(m,k)}{m^s}$ so that
\[
\frac{1}{1 + \bigl(\zeta(s) - 1\bigr)} = \sum_{k=0}^{\infty} (-1)^k \bigl(\zeta(s) -1)^k
=\sum_{k=0}^{\infty} (-1)^k \left(\sum_{m=1}^{\infty} \frac{c(m,k)}{m^s}\right),
\]
then
\[
v(n,k) = \sum_{j \leq n} c(j,k).
\]
Evaluating $p_n(x)$ at $x=0$ gives the fundamental relationship
\[
\det C_n = \sum_{i=1}^n \mu(i) = \prod_{\text{$\lambda$ nontrivial}} \!\!\!\! \lambda
= \sum_{k=0}^{\lfloor \log_2 n\rfloor}(-1)^k v(n,k),
\]
where $v(n,0)$ is defined to equal $1$.

Barrett and Jarvis~\cite{BarrettJarvis1992} showed that $C_n$
has two large real eigenvalues $\lambda_{\pm}$ satisfying
\begin{equation}\label{eqn:asyptotic}
\lambda_{\pm} = \pm \sqrt{n} + \log \sqrt{n} + \gamma - 1/2 + O\left(\frac{\log^2 n}{\sqrt{n}}\right),
\end{equation}
and that the remaining $\lfloor \log_2 n\rfloor-1$ small
nontrivial eigenvalues satisfy
\[
|\lambda| < \log_{2 -\epsilon} n
\]
for any small positive $\epsilon$ and sufficiently large $n$.
Based on numerical evidence for various values of $n$ as large
as $n=10^6$, they also made the following two-part conjecture:

\begin{conjecture}[Barrett and Jarvis~\cite{BarrettJarvis1992}]
\label{conjecture}
The small nontrivial eigenvalues $\lambda$ of $C_n$ satisfy
\begin{enumerate}
\item[(i)]  $|\lambda|<1$, and
\item[(ii)] $\re(\lambda)<1$.
\end{enumerate}
\end{conjecture}
The statement $\re(\lambda)<1$ is, of course, weaker than the
statement $|\lambda|<1$.

Vaughan~\cite{Vaughan1993} refined the asymptotic
formula~\eqref{eqn:asyptotic} for the two large eigenvalues and
showed, unconditionally, that the small eigenvalues satisfy
\[
|\lambda| \ll (\log n)^{2/5},
\]
and, upon the Riemann hypothesis, that the small eigenvalues
satisfy
\begin{equation} \label{eqn:RHbound}
|\lambda| \ll \log \log (2+n).
\end{equation}
He later showed~\cite{Vaughan1996} that $C_n$ has nontrivial
eigenvalues arbitrarily close to~$1$ for sufficiently large
$n$, suggesting that a proof of Conjecture~\ref{conjecture}
would likely be quite subtle.

Investigations of the Redheffer matrix have been extended to
group theory by Humphries~\cite{Humphries1997} and to partially
ordered sets by Wilf~\cite{Wilf2004}.

In Theorem~\ref{thm:disproof}, we resolve
Conjecture~\ref{conjecture} by showing that both parts are
false. There exist values of $n$ for which a small eigenvalue
$\lambda$ satisfies both $|\lambda|>1$ and $\re(\lambda)>1$. To
accomplish this we computed the characteristic polynomials for
$A_n$ for values of $n$ as large as $n=2^{36}$, which we
describe in~\S\ref{section:computingeigenvalues}.

\section{The determinant of $A_n$}

We now find the determinant of $A_n$.

\begin{theorem} \label{thm:determinantofAn}
Let $D_n$ be the Dirichlet matrix associated with the formal
Dirichlet series $L(s) = \sum_{k=1}^{\infty} a_k k^{-s}$ where
$a_1=1$, and write $L(s)^{-1}=\sum_{k=1}^{\infty} b_k k^{-s}$.
Let $W_n$ be the matrix whose first column is
$(0,w_2,\ldots,w_n)^T$ and whose other entries are zero. Let
$A_n = W_n+D_n$ as in~\eqref{eqn:definitionofABC}. Also, let
$\tilde{A}_n = W_n + D_n^{-1}$. Then
\begin{equation}
\det(A_n) = \sum_{k=1}^n w_k b_k \quad \text{and} \quad
\det(\tilde{A}_n) = \sum_{k=1}^n w_k a_k.
\end{equation}
\end{theorem}

\begin{corollary}
The choice $w_k=1$ produces partial sums of coefficients of
Dirichlet series:
\begin{equation}
\det A_n = \sum_{k=1}^n b_k
\quad \text{and} \quad
\det \tilde{A}_n = \sum_{k=1}^n a_k.
\end{equation}
\end{corollary}

\begin{corollary}  \label{cor:truncatedDirichletseries}
The choice $w_k = k^{-s}$ gives truncations of the Dirichlet
series $L(s)^{-1}$ and $L(s)$:
\begin{equation}
\det A_n = \sum_{k=1}^n \frac{b_k}{k^s}
\quad \text{and} \quad
\det \tilde{A}_n = \sum_{k=1}^n \frac{a_k}{k^s}.
\end{equation}
\end{corollary}
If $s$ is a complex number at which $L(s)$ and $L(s)^{-1}$
converge,
\[
\lim_{n \rightarrow \infty} \det A_n = L(s)^{-1}
\quad \text{and} \quad
\lim_{n \rightarrow \infty} \det \tilde{A}_n = L(s).
\]
So, Corollary~\ref{cor:truncatedDirichletseries} says that we
may interpret $\det A_n$ and $\det \tilde{A}_n$ as
approximating values of Dirichlet series. Since the determinant
is the product of the eigenvalues, this relates values of
Dirichlet series with eigenvalues of matrices.

\begin{proof}[Proof of Theorem~\ref{thm:determinantofAn}]
This is essentially the same argument as the one given in
Redheffer's note~\cite{Redheffer1977} where he found the
determinant of $B_n$. Since $D_n$ is upper triangular with
diagonal entry $1$, $\det D_n =\det D_n^{-1}=1$. Then
\[
\det A_n = \det D_n^{-1}\det A_n = \det D_n^{-1} \det(W_n + D_n)
=\det( D_n^{-1}W_n + I_n).
\]
The matrix $D_n^{-1} W_n$ has zeros in columns $2$ through $n$,
and its $(1,1)$-entry is $\sum_{k=2}^{n} w_k b_k$. Thus, $\det
A_n$ equals the $(1,1)$ entry of $D_n^{-1}W_n+I_n$ which is
$\sum_{k=1}^n w_k b_k$.  Replacing $D_n$ with $D_n^{-1}$ in the
argument gives $\det \tilde{A}_n = \sum_{k=1}^n w_k a_k$.
\end{proof}

\section{The characteristic polynomial of $A_n$}

The characteristic polynomial $p_n(x)=\det(I_n x-A_n)$ plays a
significant role. Previously, $p_n(x)$ was obtained for the
special case $C_n$ in~\cite{BarretForcadePollington1988}
and~\cite{Vaughan1993}. In this section, we will determine the
characteristic polynomial of the more general matrix $A_n$. The
following definition will be instrumental in describing both
the characteristic polynomial of $A_n$ and its eigenvectors.

\begin{definition}
For integers $n \geq 1$ and $k \geq 0$, we define $d(n,k)$ to
be the Dirichlet series coefficients of $\bigl(L(s)-1\bigr)^k$.
That is,
\begin{equation} \label{eqn:dnkdefinition}
\bigl( L(s)-1\bigr)^k
= \left(\sum_{k=2}^{\infty}\frac{a_n}{n^s}\right)^k
= \sum_{n=1}^{\infty} \frac{d(n,k)}{n^s}.
\end{equation}
We define $v(n,k)$ and $v_{\ell}(n,k)$ to be the weighted sums:
\begin{equation} \label{eqn:vnkdefinition}
\begin{split}
v(n,k) & = \sum_{j \leq n} w(j) d(j,k), \quad \text{and}\\
v_{\ell}(n,k) & = \sum_{j \leq n} w(j \ell) d(j,k).
\end{split}
\end{equation}
\end{definition}

Several cases of this definition are important to keep in mind:
$d(1,0)=1$ and $d(n,0)=0$ for $n>1$; also, both $d(n,k)$ and
$v(n,k)$ are zero if $n<2^k$ since a number smaller than $2^k$
cannot be written as a product of $k$ nontrivial factors.

From the definition of $d(n,k)$,
\[
\sum_{n=1}^{\infty} \frac{d(n,k)}{n^s}
=
\left(\sum_{k=2}^{\infty}\frac{a_n}{n^s}\right)\left(\sum_{k=2}^{\infty}\frac{a_n}{n^s}\right)^{k-1}
=
\left(\sum_{k=2}^{\infty}\frac{a_n}{n^s}\right)\left(\sum_{n=1}^{\infty} \frac{d(n,k-1)}{n^s}\right),
\]
which immediately gives the elementary recurrence relation:
\begin{lemma}  \label{lemma:dnkrecurrence}
If $k \geq 1$, then
\begin{equation} \label{eqn:dnkrecurrence}
d(n,k) = \sum_{\substack{i|n \\ 1< i}} a(i) d(n/i,k-1)
=\sum_{\substack{j|n \\ j<n}} a(n/j)d(j,k-1).
\end{equation}
\end{lemma}

\begin{theorem} \label{thm:CharacteristicPolynomial}
The characteristic polynomial $p_n(x) = \det(xI_n - A_n)$ is
\begin{equation}\label{eqn:CharacteristicPolynomial}
p_n(x) = (x-1)^{n-r-1} \Big((x-1)^{r+1}-\sum_{k=1}^r v(n,k) (x-1)^{r-k}\Big),
\end{equation}
where $r=\lfloor \log_2 n \rfloor$. Consequently, if
$v(n,r)\not=0$, the algebraic multiplicity of the trivial
eigenvalue $\lambda =1$ is $n-r-1$.
\end{theorem}

\begin{proof}
We will use the cofactor expansion to calculate the
characteristic polynomial $p_n(x)=\det(xI_n-A_n)$.  Write
$M_n=xI_n-A_n$ and let
\[
M_n({i_1,\ldots,i_s} \mid {j_1,\ldots,j_{t}})
\]
denoted the matrix obtained by removing the rows indexed by
$i_1,\ldots,i_s$ and the columns indexed by $j_1,\ldots,j_t$
from $M_n$. The cofactor expansion of the determinant along the
first column is
\begin{equation*}
p_n(x) = (x-1)^n + \sum_{k=2}^n (-1)^{k} w_k \det M_n(k \mid 1).
\end{equation*}

The matrix $M_n(k \mid 1)$ is a block matrix whose upper left
$(k-1)\times(k-1)$ block is $M_k(k \mid 1)$, whose lower left
$(n-k)\times(k-1)$ block is zero, and whose lower right
$(n-k)\times(n-k)$ block is upper triangular with diagonal
entries $x-1$. Thus
\begin{equation*}
\det
M_n(k \mid 1)=(x-1)^{n-k} \det M_k(k \mid 1),
\end{equation*}
where we understand $\det M_1(1 \mid 1)$ to be $1$, and
\begin{equation} \label{eqn:charpoly1}
p_n(x) = (x-1)^n + \sum_{k=2}^n (-1)^{k} w_k (x-1)^{n-k}\det M_k(k \mid 1).
\end{equation}

The
$\ell$th entry in the last column of $M_k(k \mid 1)$ is
$-a_{n/\ell}$ if $\ell$ divides $k$; otherwise, it is zero.
Then the cofactor expansion of $\det M_k(k \mid 1)$ along the
last column is
\begin{equation*}
\det M_k(k \mid 1) = \sum_{\substack{\ell \mid k \\ \ell<k}}
(-1)^{k+\ell} a_{k/\ell} \det M_k(\ell,k \mid 1,k).
\end{equation*}
The matrix $M_k(\ell,k \mid 1, k)$ is also a block matrix.
Since the upper left $(\ell-1)\times(\ell-1)$ block is
$M_{\ell}(\ell
\mid 1)$, the lower left $(k-\ell-1) \times(\ell-1)$ block is zero,
and the lower right $(k-\ell-1)\times(k-\ell-1)$ block is upper
triangular with diagonal entries $x-1$,
\begin{equation*}
\det M_k(\ell,k \mid 1,k) = (x-1)^{k-\ell-1} \det M_\ell(\ell \mid 1).
\end{equation*}
This shows that
\begin{equation} \label{eqn:charpoly2}
\det M_k(k \mid 1) = \sum_{\substack{\ell \mid k \\ \ell<k}}
(-1)^{k+\ell} a_{k/\ell} (x-1)^{k-\ell-1}\det M_{\ell}(\ell \mid 1).
\end{equation}
In other words, the quantity $q_k(x)=(-1)^{k-1}M_k(k \mid 1)$
satisfies the recurrence relation:
\begin{equation}   \label{eqn:charpoly3}
\begin{split}
q_1(x) & = 1, \\
q_k(x) & = \sum_{\substack{\ell \mid k \\ \ell < k}} a_{k/\ell}
(x-1)^{k-\ell-1} q_{\ell}(x) \quad \text{for $k > 1$.}
\end{split}
\end{equation}

On the other hand, consider the polynomial $t_{\ell}(x)$
defined by
\begin{equation} \label{eqn:tsubk}
t_{\ell}(x)=\sum_{j \geq 0} d(\ell,j) (x-1)^{\ell-j-1}.
\end{equation}
Then $t_{1}(x)=1$. For $\ell>1$, the term in the sum
corresponding to $j=0$ is zero since $d(\ell,0)=0$ in that
case. For $k>1$, calculating the right hand side
of~\eqref{eqn:charpoly3} with $t_{\ell}(x)$ in place of
$q_{\ell}(x)$ and applying Lemma~\ref{lemma:dnkrecurrence}
gives
\begin{align*}
\sum_{\substack{\ell \mid k \\ \ell < k}} a_{k/\ell}
(x-1)^{k-\ell-1} t_{\ell}(x)
& =
\sum_{\substack{\ell \mid k \\ \ell < k}}
\sum_{j \geq 0} a_{k/\ell} d(\ell,j)(x-1)^{k-j-2}  \\
& =
\sum_{j \geq 0} d(k,j+1)(x-1)^{k-j-2}  \\
& =
\sum_{j \geq 1} d(k,j) (x-1)^{k-j-1} \\
& = t_k(x).
\end{align*}
Since $t_k(x)$ and $q_k(x)$ both satisfy the same recurrence
relations, they are equal.  This shows that
\[
(-1)^{k-1}M_k(k \mid 1) = q_k(x)=
t_{k}(x)=\sum_{j \geq 0} d(k,j) (x-1)^{k-j-1}.
\]
Substituting the last expression into~\eqref{eqn:charpoly1}
gives
\begin{align*}
p_n(x)
& = (x-1)^n - \sum_{k=2}^n \sum_{j \geq 1} w_k d(k,j)(x-1)^{n-j-1}\\
& = (x-1)^n - \sum_{j \geq 1} v(n,j) (x-1)^{n-j-1}.
\end{align*}
Since $v(n,j)=0$ for $j>r=\lfloor \log_2(n) \rfloor$, this is
\begin{align*}
p_n(x)
& = (x-1)^n - \sum_{j=1}^{r} v(n,j) (x-1)^{n-j-1} \\
& = (x-1)^{n-r-1}\Big( (x-1)^{r+1}-\sum_{j=1}^r v(n,j)(x-1)^{r-j}\Big),
\end{align*}
which proves the theorem.
\end{proof}

\section{The eigenvectors of $A_n$}

\begin{theorem}\label{thm:eigenvectors}
Let $\lambda \not=1$ be a nontrivial eigenvalue of $A_n$. Then
$\lambda$ is a simple eigenvalue, and a basis for the one
dimensional eigenspace of $A_n$ associated with $\lambda$ is
the vector
\[
u = \bigl[ \lambda-1,
X_{2}(\lfloor n/2\rfloor),
X_{3}(\lfloor n/3\rfloor),
X_{4}(\lfloor n/4\rfloor),\ldots,
X_{n}(\lfloor n/n\rfloor) \bigr]^{T}
\]
where
\[
X_{j}(q)  =
\sum_{k \geq 0} \frac{v_{j}(q, k)}{(\lambda-1)^k} =
1 + \frac{v_j(q,1)}{\lambda-1} + \frac{v_j(q,2)}{(\lambda-1)^2}
+\frac{v_j(q,3)}{(\lambda-1)^3}+\cdots
\]
\end{theorem}

\begin{proof}[Proof of Theorem~\ref{thm:eigenvectors}]
For $i \geq 2$, the $i$th entry of $A_n u$ is
\begin{align*}
(A_n u)_i
& =
w_i (\lambda -1) + \sum_{1 \leq \ell \leq n/i} a_{\ell} u_{\ell i} \\
&  =
w_i(\lambda -1) + \sum_{1 \leq \ell \leq n/i} a_{\ell} X_{\ell i}(\lfloor n/(\ell i)\rfloor) \\
& =
w_i (\lambda-1) + \sum_{1 \leq \ell \leq n/i}a_{\ell} \sum_{\substack{k \geq 0\\ 1 \leq m \leq n/i}}
w(i \ell m) d(m,k)(\lambda-1)^{-k}   \\
& =
w_i (\lambda-1) + \sum_{k \geq 0} \Big(
\sum_{\substack{1 \leq \ell \leq n/i \\ 1 \leq m \leq n/(\ell i)}}
a(\ell) w(i \ell m) d(m,k) \Big) (\lambda-1)^{-k}  \\
& =
w_i (\lambda-1) + \sum_{k \geq 0} \Big(
\sum_{1 \leq t \leq n/i} w(it) \sum_{s \mid t} a(t/s) d(s,k)\Big)(\lambda-1)^{-k} \quad\text{[set $t=i \ell$]}\\
& =
w_i(\lambda-1) + \sum_{k \geq 0} \sum_{1 \leq t \leq n/i} w(it)\bigl(d(t,k) + d(t,k+1)\bigr) (\lambda-1)^{-k} \quad\text{[by~\eqref{eqn:dnkrecurrence}]}\\
& =
\sum_{k \geq 0} v_i(\lfloor n/i \rfloor,k)(\lambda-1)^{-k}
+w_i(\lambda-1) +\sum_{k \geq 1} v_i(\lfloor n/i \rfloor, k)(\lambda-1)^{-k+1} \\
& =
\sum_{k \geq 0} v_i(\lfloor n/i \rfloor,k)(\lambda-1)^{-k} + (\lambda-1)\sum_{k \geq 0} v_i(\lfloor n/i \rfloor,k)(\lambda-1)^{-k} \\
& = \lambda \sum_{k \geq 0} v_i(\lfloor n/i \rfloor,k)(\lambda-1)^{-k} \\
& = \lambda X_i( \lfloor n/i \rfloor) \\
& = \lambda u_i.
\end{align*}
In the calculation for $(A_n u)_i$ with $i \geq 2$, the term
$a_i u_{\ell i}$ when $\ell=1$ was equal to $a_i X_i(\lfloor
n/i\rfloor)$, but this term should be omitted from the case
$i=1$. Taking this into account and going to the second to last
step of the previous calculation gives
\begin{align*}
(A_n u)_1
& =
\lambda X_1(n) - X_1(n) \\
& = (\lambda -1)\Bigl(1+\sum_{k \geq 1} v(n,k) (\lambda-1)^{-k}\Bigr) \\
& = (\lambda -1)\bigl[ 1+(\lambda-1)\bigr] \quad \text{by Theorem~\ref{thm:CharacteristicPolynomial}} \\
& = \lambda(\lambda-1) \\
& = \lambda u_1.
\end{align*}
This shows that the vector $u$ is a nonzero eigenvector for
$\lambda$.

To see why the eigenspace of $\lambda$ is one-dimensional,
consider the submatrix of $A_n - \lambda I$ obtained by
deleting the first row and column. This $(n-1)\times(n-1)$
matrix is upper triangular with nonzero entries on the
diagonal. Hence, it is invertible implying that the rank of
$A_n - \lambda I_n$ is $\geq n-1$. Since we found a nontrivial
eigenvector, the nullity is $\geq 1$. So, the nullity of
$A_n-\lambda I$ must be exactly one. This completes the proof.
\end{proof}

\begin{theorem}\label{thm:eigenvectorstranspose}
Let $\lambda \not=1$ be a nontrivial eigenvalue of $A_n$. A
basis for the one dimensional eigenspace of $A_n^T$ associated
with $\lambda$ is the vector
\begin{equation} \label{eqn:eigenvectortranspose}
v = \bigl[ 1,Y_{\lambda}(2),Y_{\lambda}(3),\ldots,Y_{\lambda}(n)\bigr]^T.
\end{equation}
where
\[
Y_{\lambda}(q) = \sum_{k \geq 0} \frac{d(q,k)}{(\lambda-1)^k} =
d(q,0)+\frac{d(q,1)}{\lambda-1}+\frac{d(q,2)}{(\lambda-1)^2}
+\cdots.
\]
\end{theorem}
Interestingly, the algebraic expression for $v$ does not
explicitly rely on the symbols $w_2,\ldots,w_n$ in the first
column of $A_n$. However, altering $w_2,\ldots,w_n$ changes the
possible numeric values of $\lambda$.

\begin{proof}[Proof of Theorem \ref{thm:eigenvectorstranspose}]
For $i \geq 2$,the $i$th entry of $A_n^T v$ is
\begin{align*}
(A_n^T v)_i
& =
\sum_{\ell \mid i} a(i/\ell) Y(\ell) \\
& =
\sum_{k \geq 0} \sum_{\ell \mid i} a(i/\ell) d(\ell,k)(\lambda-1)^{-k} \\
& =
\sum_{k \geq 0} \bigl[ d(i,k) + d(i,k+1) \bigr] (\lambda-1)^{-k} \quad \text{by \eqref{eqn:dnkrecurrence}} \\
& =
Y(i) + (\lambda-1) \sum_{k \geq 1} d(i,k)(\lambda-1)^{-k} \\
& =
Y(i) + (\lambda-1) Y(i) \quad \text{[since $d(i,0)=0$]}\\
& =
\lambda Y(i).
\end{align*}
The first entry of $A_n^T v$ is
\begin{align*}
(A_n^T)_i
& =
\sum_{1 \leq j \leq n} w_j Y(j) \\
& =
\sum_{k \geq 0} \sum_{1 \leq j \leq n} w_j d(j,k) (\lambda-1)^{-k} \\
& =
\sum_{k \geq 0} v(n,k) (\lambda-1)^{-k} \\
& =
1+ \sum_{k \geq 1} v(n,k)(\lambda-1)^{-k} \\
& =
1+(\lambda-1) \quad \text{[by Theorem~\ref{thm:CharacteristicPolynomial}]} \\
& = \lambda v_1.
\end{align*}
This shows that $v=[Y(1),\ldots,Y(n)]^T$ is a nonzero
eigenvector of $A_n^T$. The dimension of the eigenspace is one,
as explained in the proof of Theorem~\ref{thm:eigenvectors}.
\end{proof}

\section{Computing eigenvalues of $C_n$ for large $n$}
\label{section:computingeigenvalues}

Theorem~\ref{thm:CharacteristicPolynomial} expresses the
characteristic polynomial of the matrix $A_n$ in terms of the
numbers $v(n,k)$. In this section, we will explain how to
explicitly calculate the characteristic polynomial $p_n(x)$ for
large values of $n$ for the special case $C_n$ in which
$w_i=a_i=1$ for all $i$. The method given below in
Theorem~\ref{thm:vnkrecursion2} was used to find $p_n(x)$ for
$n$ as large as $n=2^{36}$ in a few hours on a desktop
computer. To accomplish this, it is necessary to use a more
efficient algorithm for finding the coefficients than a brute
force approach based directly on the definition of matrix
$C_n$.  Even with Theorem~\ref{thm:CharacteristicPolynomial} we
need a better method for computing $v(n,k)$ than the direct
application of the definition of $v(n,k)$
in~\eqref{eqn:vnkdefinition}.

\begin{lemma}
Suppose $a_{\ell}=w_\ell=1$ for all $\ell$. If $1 \leq 2^k \leq
n$, then
\begin{equation}\label{eqn:vnkrecursion1}
v(n,k) = \sum_{i>1}v\bigl( \lfloor \tfrac{n}{i} \rfloor, k-1\bigr)
=
\sum_{j<n} \big(\big\lfloor\tfrac{n}{j}\big\rfloor-\big\lfloor\tfrac{n}{j+1}\big\rfloor\big)v(j,k-1).
\end{equation}
\end{lemma}

If both $a_k=1$ and $w_k=1$ for all $k$, then $v(n,k)$
represents the number of ways to form products of $k$
nontrivial factors whose product is $\leq n$ and where order
matters. In this case, $v(n,k)$ represents a count of lattices
points in $k$-dimensional space:
\begin{equation} \label{eqn:vnkgeometric}
v(n,k) = \big|\{(\ell_1,\ldots,\ell_k) \in \Z^k \, : \;
\text{$\ell_1 \ell_2 \cdots \ell_k \leq n$ and $\ell_i \geq 2$ for
all $i$} \}\big|.
\end{equation}

\begin{proof}
The first equality in~\eqref{eqn:vnkrecursion1} is evident
from~\eqref{eqn:vnkgeometric} by letting one component of
$(\ell_1,\ldots,\ell_k)$, say $\ell_k$, be the index of
summation $i$. The second equality in~\eqref{eqn:vnkrecursion1}
is obtained by re-indexing the sum over the distinct values of
$j=\lfloor n/i\rfloor$. For a given positive integer $j$,
\[
j = \left\lfloor \frac{n}{i}\right\rfloor
\quad \Leftrightarrow \quad
j \leq \frac{n}{i} < j+1
\quad \Leftrightarrow \quad
\frac{n}{j+1} < i \leq \frac{n}{j}.
\]
Thus, the number of distinct $i$ for which $\lfloor n/i \rfloor
= j$ is $\lfloor \tfrac{n}{j}\rfloor -\lfloor
\tfrac{n}{j+1}\rfloor$.
\end{proof}

The first recursion formula in~\eqref{eqn:vnkrecursion1} is
computationally inefficient since there can be many distinct
values of $i_1$ and $i_2$ such that $\lfloor n/i_1 \rfloor =
\lfloor n/i_2\rfloor$. The second is inefficient since there
can be many values of $j$ such that  $\lfloor n/j \rfloor -
\lfloor n/(j+1)\rfloor$ is zero.
The next theorem helps to remove this redundancy by rewriting
the summation to have significantly fewer terms.

\begin{theorem} \label{thm:vnkrecursion2}
Assume $a_{\ell}=w_{\ell}=1$ for all $\ell$. Suppose $1 \leq
2^k \leq n$ and $k \geq 1$. Then
\begin{equation} \label{eqn:vnkrecursion2}
v(n,k) =
\sum_{i=2}^{s}
v\bigl( \bigl\lfloor \tfrac{n}{i} \bigr\rfloor, k-1\bigr)
+
\sum_{j=2^{k-1}}^{\lfloor \sqrt{n}\rfloor}
\bigl(\bigl\lfloor \tfrac{n}{j}\bigr\rfloor -
\bigl\lfloor \tfrac{n}{j+1}\bigr\rfloor  \bigr)\,v( j, k-1),
\end{equation}
where
$s =\bigl\lfloor\frac{n}{\lfloor\sqrt{n}\rfloor+1}\bigr\rfloor$.
\end{theorem}

\begin{proof}
This argument applies the hyperbola method from analytic number
theory. Rewrite~\eqref{eqn:vnkrecursion1} as
\begin{equation} \label{eqn:twotermsum}
v(n,k) =
\sum_{\lfloor n/i \rfloor \geq \lfloor\sqrt{n}\rfloor+1}
\!\!\!\!\! v\big(\big\lfloor \tfrac{n}{i} \big\rfloor, k-1\big)
+
\sum_{\lfloor n/i \rfloor  \leq \lfloor\sqrt{n}\rfloor}
\!\!\! v\bigl( \bigl\lfloor \tfrac{n}{i} \bigr\rfloor, k-1\bigr),
\end{equation}
where the index $i$ in each summation satisfies $2 \leq i \leq
\lfloor n/2^{k-1}\rfloor$.
In the first summation, since both $i$ and
$\lfloor\sqrt{n}\rfloor+1$ are integers,
\[
\left\lfloor\frac{n}{i}\right\rfloor \geq \lfloor \sqrt{n}\rfloor+1
\ \Leftrightarrow \
\frac{n}{i} \geq \lfloor\sqrt{n}\rfloor+1
\ \Leftrightarrow \
i \leq \frac{n}{\lfloor\sqrt{n}\rfloor+1}
\ \Leftrightarrow \
i \leq \left\lfloor \frac{n}{\lfloor\sqrt{n}\rfloor+1}\right\rfloor.
\]
This gives the value $s=\lfloor n /(\lfloor
\sqrt{n}\rfloor+1)\rfloor$ in the first summation
in equation~\eqref{eqn:vnkrecursion2}. In the second summation
in~\eqref{eqn:twotermsum}, we re-index the sum over the
distinct values of $j=\lfloor
n/i\rfloor\leq\lfloor\sqrt{n}\rfloor$. For a given positive
integer $j$,
\[
j = \left\lfloor \frac{n}{i}\right\rfloor
\quad \Leftrightarrow \quad
j \leq \frac{n}{i} < j+1
\quad \Leftrightarrow \quad
\frac{n}{j+1} < i \leq \frac{n}{j}.
\]
Thus, the number of distinct $i$ for which $\lfloor n/i \rfloor
= j$ is $\lfloor \tfrac{n}{j}\rfloor -\lfloor
\tfrac{n}{j+1}\rfloor$, allowing the second summation
in~\eqref{eqn:twotermsum} to be written as
\[
\sum_{j=2^{k-1}}^{\lfloor \sqrt{n}\rfloor}
\bigl(\bigl\lfloor \tfrac{n}{j}\bigr\rfloor -
\bigl\lfloor \tfrac{n}{j+1}\bigr\rfloor  \bigr)\,v( j, k-1).
\]
This proves~\eqref{eqn:vnkrecursion2}.
\end{proof}
It is interesting to note that $s$ in
Lemma~\ref{thm:vnkrecursion2} is equal to either
$\lfloor\sqrt{n}\rfloor$ or $\lfloor\sqrt{n}\rfloor-1$
according to
\[
s =
\left\lfloor\frac{n}{\lfloor\sqrt{n}\rfloor+1}\right\rfloor
=
\begin{cases}
\lfloor\sqrt{n}\rfloor & \text{if $n-\lfloor\sqrt{n}\rfloor^2 \geq \lfloor\sqrt{n}\rfloor$}, \\
\lfloor\sqrt{n}\rfloor-1 & \text{if $n-\lfloor\sqrt{n}\rfloor^2 < \lfloor\sqrt{n}\rfloor$}.
\end{cases}
\]

\begin{theorem} \label{thm:disproof}
Conjecture~\ref{conjecture} is false. There exist values of $n$
for which a small eigenvalue $\lambda$  of $C_n$ satisfies both
$|\lambda|>1$ and $\re(\lambda)>1$.
\end{theorem}

\begin{proof}
The characteristic polynomial $p_n(x)$ of the general matrix
$A_n$ was given in Theorem~\ref{thm:CharacteristicPolynomial}.
By implementing the recursive formula in
Theorem~\ref{thm:vnkrecursion2}, we were able to calculate the
characteristic polynomial for the special case $C_n$ for
relatively large values of $n$, such as $n=2^{36}$, within a
few hours on a desktop computer.

A table listing the maximum absolute value and real part of
small nontrivial eigenvalues of $C_n$ for $n=10^6$ and $n=2^r$
with $28 \leq r \leq 36$ is given below:
\[
\begin{array}{lcr|l|l}
n       &   &           & \max\{|\lambda|\} & \max\{\re(\lambda)\}   \\ \hline\hline
10^6    & = & 1,000,000   & 0.983108 & 0.983108 \\
 2^{28} & = & 268,435,456 & 0.998885 & 0.998739 \\
 2^{29} & = & 536,870,912 & 0.999120 & 0.998989 \\
 2^{30} & = & 1,073,741,824 & 0.999324 & 0.999206 \\
 2^{31} & = & 2,147,483,648 & 0.999501 & 0.999395 \\
 2^{32} & = & 4,294,967,296 & 0.999676 & 0.999560 \\
 2^{33} & = & 8,589,934,592 &  1.002646 & 0.999704 \\
 2^{34} & = & 17,179,869,184 & 1.005213 & 0.999829 \\
 2^{35} & = & 34,359,738,368 & 1.007423 & 0.999939 \\
 2^{36} & = & 68,719,476,736 & 1.031192 & 1.000036
\end{array}
\]
The example with $n=2^{36}$ provides a counter-example to both
parts of Conjecture~\ref{conjecture}. A sample of the
coefficients $v(n,k)$ of $p_n(x)$ for $n=10^6$, $n=2^{28}$, and
$n=2^{36}$ is given in Table~\ref{table:vnk}.
\end{proof}

\pagebreak

\begin{table}[!h]
\caption{Values of $v(n,k)$ for $n=10^{6}$, $n=2^{28}$, and $n=2^{36}$ }
\label{table:vnk}
\begin{tabular}{rrrr}
$k$ & $v(10^6,k)$ & $v(2^{28}, k)$ & $v(2^{36},k)$ \\ \hline
 1 & 999999 & 268435455 & 68719476735 \\
 2 & 11970035 & 4714411991 & 1587951104025 \\
 3 & 67120491 & 39550266080 & 17712699735807 \\
 4 & 233959922 & 210866000001 & 127006997038631 \\
 5 & 567345854 & 801946179797 & 657738684402616 \\
 6 & 1015020739 & 2314766752399 & 2620541404211325 \\
 7 & 1386286166 & 5267935378357 & 8354699452581663 \\
 8 & 1475169888 & 9693670870002 & 21888970237054221 \\
 9 & 1237295133 & 14675212443928 & 48028484118248949 \\
 10 & 822451796 & 18500845515388 & 89496511738284187 \\
 11 & 433656192 & 19585798031078 & 143118705146069804 \\
 12 & 180821164 & 17506983509953 & 197979547265239162 \\
 13 & 59146673 & 13254336924806 & 238336089820847725 \\
 14 & 14935574 & 8508754910066 & 250812663743567239 \\
 15 & 2829114 & 4628591443629 & 231467885026020936 \\
 16 & 383693 & 2128656115076 & 187727209728498411 \\
 17 & 34630 & 824357770148 & 133949812310943213 \\
 18 & 1672 & 267263904116 & 84103735312636462 \\
 19 & 20 & 71941723387 & 46433832280215021 \\
 20 &  & 15889930335 & 22505741596654059 \\
 21 &  & 2830811858 & 9551600816612963 \\
 22 &  & 396537923 & 3536981261202340 \\
 23 &  & 42162106 & 1137490727898326 \\
 24 &  & 3284753 & 315879734318303 \\
 25 &  & 177731 & 75228001661856 \\
 26 &  & 4707 & 15244074212812 \\
 27 &  & 55 & 2604780031507 \\
 28 &  & 1 & 371154513760 \\
 29 &  &  & 43388420848 \\
 30 &  &  & 4049932603 \\
 31 &  &  & 290175811 \\
 32 &  &  & 15487073 \\
 33 &  &  & 582143 \\
 34 &  &  & 9555 \\
 35 &  &  & 71 \\
 36 &  &  & 1
\end{tabular}
\end{table}

\section{Acknowledgment}
The author thanks Wayne Barrett who introduced him to the
problem discussed in this paper and for extensive conversations
about several of the results described here. Also the author
thanks Rodney Forcade for several helpful suggestions.

\end{document}